\documentclass[11pt]{amsart}

\usepackage[margin=1.2in]{geometry}
\usepackage{amsmath,amssymb,amsthm,url,mathtools}

\title{Arithmetic functions at consecutive shifted primes}

\author{Paul Pollack}
\address{Department of Mathematics\\University of Georgia\\Athens, GA 30602}
\email{pollack@uga.edu}

\author{Lola Thompson}
\address{Department of Mathematics \\Oberlin College\\Oberlin, OH 44074}
\email{lola.thompson@oberlin.edu}

\DeclareMathAlphabet{\curly}{U}{rsfs}{m}{n}

\newtheorem{thm}{Theorem}[section]

\newtheorem{lem}[thm]{Lemma}

\newtheorem{prop}[thm]{Proposition}

\newtheorem*{questions}{Questions}
\newtheorem*{theorem*}{Theorem}
\theoremstyle{remark}\newtheorem*{remark}{Remark}

\numberwithin{equation}{section}

\newcommand{\bad}{\mathrm{bad}}
\newcommand{\diff}{\mathrm{d}}

\newcommand\Ii{\mathcal{I}}

\newcommand\N{\mathbf{N}}

\newcommand\Hh{\mathcal{H}}
\newcommand\lcm{\mathrm{lcm}}

\newcommand\Pp{\curly{P}}

\newcommand\Prob{\mathbf{Pr}}
\newcommand\E{\mathbf{E}}

\renewcommand{\phi}{\varphi}
\makeatletter
\renewcommand{\pod}[1]{\mathchoice
  {\allowbreak \if@display \mkern 18mu\else \mkern 8mu\fi (#1)}
  {\allowbreak \if@display \mkern 18mu\else \mkern 8mu\fi (#1)}
  {\mkern4mu(#1)}
  {\mkern4mu(#1)}
}

\begin{document}

\begin{abstract} For each of the functions $f \in \{\phi, \sigma, \omega, \tau\}$ and every natural number $K$, we show that there are infinitely many solutions to the inequalities $f(p_n-1) < f(p_{n+1}-1) < \dots < f(p_{n+K}-1)$, and similarly for $f(p_n-1) > f(p_{n+1}-1) > \dots > f(p_{n+K}-1)$. We also answer some questions of Sierpi\'nski on the digit sums of consecutive primes. The arguments make essential use of Maynard and Tao's method for producing many primes in intervals of bounded length.
\end{abstract}


\dedicatory{In memory of Paul Bateman and Heini Halberstam}
\maketitle

\section{Introduction}
The study of prime numbers goes back thousands of years and contains many deep and beautiful results. Yet to this day, we have few theorems concerning properties of \emph{consecutive primes}. Perhaps the exception that proves the rule is Shiu's theorem \cite{shiu00} that each coprime progression $a\bmod{q}$ contains arbitrarily long runs of consecutive primes. Results like Shiu's have been thin on the ground for a very simple reason; until extremely recently, analytic number theory had very few tools that could be applied to these problems.

This situation has changed dramatically with the recent breakthrough of Maynard and Tao on bounded gaps between primes. The first to notice this paradigm shift were Banks--Freiberg--Turnage-Butterbaugh \cite{BFTB14}, who gave a strikingly simple re-proof of Shiu's result using the Maynard--Tao method. They also showed how the same circle of ideas leads to the solution of certain longstanding problems of Erd\H{o}s and Tur\'an about increasing and decreasing runs in the sequence of prime gaps. See also \cite{pollack14}, where under the Generalized Riemann Hypothesis, it is shown that the Maynard--Tao method produces arbitrarily long runs of consecutive primes possessing a given primitive root.

The objective of the present paper is to demonstrate how the Maynard--Tao work can be leveraged to study the behavior of arithmetic functions at consecutive shifted primes. Let $p_1 < p_2 < p_3 < \dots$ be the sequence of primes in the usual increasing order. The bulk of the paper is devoted to a proof of the following theorem.

\begin{thm}\label{thm:main} Let $f$ be any of the Euler $\phi$-function, the sum-of-divisors function $\sigma$, the function $\omega$ counting the number of distinct prime factors, or the count-of-all-divisors function $\tau$. We prove that for every $K$, there are infinitely many solutions $n$ to the inequalities
\[ f(p_n-1) < f(p_{n+1}-1) < \dots < f(p_{n+K-1}-1), \]
and similarly with all of the inequalities reversed.
\end{thm}

We conclude by discussing a function of a rather different nature than those considered in Theorem \ref{thm:main}, namely the function $s_g(n)$ giving the sum of the  base-$g$ digits of $n$. Since there is at most one prime $p$ ending in the digit $0$, the problems we are interested in are the same if we consider $s_g(p)$ or $s_g(p-1)$. We will stick to the former, both for reasons of notational simplicity and for historical precendent.

The following questions were posed by Sierpi\'nski \cite{sierpinski61} in a four-page 1961 paper. He worked primarily with $g=10$, but we formulate the questions more generally.

\begin{questions}\label{conj:constant} Are there arbitrarily long runs of consecutive primes $p$ on which $s_g(p)$ is constant? increasing? decreasing?
\end{questions}


In this direction, Sierpi\'nski noted that $s_{g}(p_n) < s_{g}(p_{n+1})$ infinitely often, as a consequence of the easy fact that $s_g(p_n)$ is unbounded. A year later, Erd\H{o}s \cite{erdos62} showed the (much more difficult) complementary result that $s_{g}(p_n) > s_{g}(p_{n+1})$ infinitely often. Neither of these authors could prove a corresponding result for  $p_n, p_{n+1}$, and $p_{n+2}$. However, assuming Dickson's prime $k$-tuples conjecture, Sierpi\'nski (op. cit.) showed that infinitely often $s_{10}(p_n) > s_{10}(p_{n+1}) > s_{10}(p_{n+2})$. Under the same hypothesis, he also showed that $s_{10}(p_n) = s_{10}(p_{n+1})$ infinitely often. At the end of his article, Sierpi\'nski communicates a claim of Schinzel that, assuming ``Hypothesis H'' (described in \cite{SS58}), all three questions have an affirmative answer. (Schinzel's argument appears to have never been published.) We prove this without any unproved hypothesis in the final section of this paper.

\subsection*{Notation} The letter $p$ always represents a prime. Similarly, $p_n$ always represents the $n$th term in a sequence of primes, but \textbf{not} always the $n$th term in the full sequence of primes.  We use $\log_k$ for the $k$-fold iterated logarithm.  An example of nonstandard notation is our use of $\varrho(n)\coloneqq \sum_{p^e\parallel n} e$ for the count of primes dividing $n$ with multiplicity. (The more standard notation, $\Omega(n)$, conflicts with our notation for the Maynard--Tao sample space.)

\section{Preliminaries}

\subsection{Review of the Maynard--Tao approach} Since our work relies crucially on the method of Maynard and Tao, we say a few words about the basic set up. Let $\Hh\coloneqq\{h_1, h_2, \dots,  h_k\}$ denote a fixed \emph{admissible $k$-tuple}, i.e., a list of $k$ distinct integers that does not occupy all of the residue classes modulo $p$ for any prime $p$. The Hardy--Littlewood prime $k$-tuples conjecture then predicts that there are infinitely many $n$ for which the shifted tuple $n+h_1, \dots, n+h_k$ consists entirely of primes. This is currently beyond reach, but the Maynard--Tao work allows us to find values of $n$ for which the tuple $n+h_1, \dots, n+h_k$ contains several prime elements.

We let $N$ be large, and we look for our $n$ in the dyadic interval $[N,2N)$. We consider only $n$ satisfying a mild congruence condition. Let \[ W\coloneqq \prod_{p \leq \log_3{N}}p, \] and choose an integer $\nu$ so that
\begin{equation}\label{eq:basicnu} \gcd(\nu+h_i,W)=1 \quad\text{for all $1 \leq i \leq k$}; \end{equation} the existence of such a $\nu$ follows from the admissibility of $\Hh$. We will restrict attention to $n$ satisfying $n \equiv \nu\pmod{W}$. For our purposes, one should think of this condition on $n$ as a sort of ``pre-sieving'' --- we are allowed to specify in advance the mod $p$ behavior of $n$ at all small primes $p$.
	
Let $w(n)$ denote nonnegative weights (to be chosen momentarily), and let $\chi_{\Pp}$ denote the characteristic function of the set $\Pp$ of prime numbers. We consider the sums
\[ S_1\coloneqq \sum_{\substack{N \leq n < 2N \\ n \equiv \nu\pmod{W}}} w(n)\quad\text{and}\quad S_2\coloneqq \sum_{\substack{N \leq n < 2N \\ n\equiv \nu\pmod{W}}} \left(\sum_{i=1}^{k} \chi_{\Pp}(n+h_i)\right)w(n). \]
The fraction $S_2/S_1$ is a weighted average of the number of primes among $n+h_1, \dots, n+h_k$, where the average is taken over the sample space
\[ \Omega\coloneqq \{N \le n < 2N: n \equiv \nu\pmod{W}\}. \]
Consequently, if $S_2 > (K-1) S_1$ for the positive integer $K$, then at least $K$ of the numbers $n+h_1, \dots, n+h_k$ are primes, for some $n \in \Omega$.

For this strategy to bear fruit, one needs to select weights $w(n)$ so that the sums $S_2$ and $S_1$ can be estimated by the tools of asymptotic analysis, and so that the resulting ratio $S_2/S_1$ is large. Finding such a choice of $w(n)$ is the main innovation
in the Maynard--Tao method.  The following is a restatement of \cite[Proposition 4.1]{maynard14}.

\begin{prop}\label{prop:main-maynard} Let $\theta$ be a positive real number with $\theta < \frac{1}{4}$.
Let $F$ be a piecewise differentable function supported on the simplex $\{(x_1, \dots, x_k): \text{each }x_i \geq 0, \sum_{i=1}^{k} x_i \leq 1\}$. With $R\coloneqq N^{\theta}$, put
\begin{equation}\label{eq:lambdadef} \lambda_{d_1, \dots, d_k}\coloneqq \left(\prod_{i=1}^{k} \mu(d_i) d_i\right) \sum_{\substack{r_1, \dots, r_k \\ d_i \mid r_i\,\forall i \\ (r_i, W)=1\,\forall i}} \frac{\mu(\prod_{i=1}^{k} r_i)^2}{\prod_{i=1}^{k} \phi(r_i)} F\left(\frac{\log{r_1}}{\log{R}},  \dots, \frac{\log{r_k}}{\log{R}}\right)\end{equation}
whenever $\gcd(\prod_{i=1}^{k} d_i,W)=1$, and let $\lambda_{d_1, \dots, d_k}=0$ otherwise. Let
\begin{equation}\label{eq:wndef} w(n)\coloneqq \left(\sum_{d_i \mid n+h_i\,\forall i} \lambda_{d_1, \dots, d_k}\right)^2. \end{equation}
Then as $N\to\infty$,
\begin{align*} S_1 &\sim \frac{\phi(W)^k}{W^{k+1}} N (\log{R})^k \cdot I_k(F), ~~\text{and} \\
S_2 &\sim \frac{\phi(W)^k}{W^{k+1}} \frac{N}{\log{N}} (\log{R})^{k+1} \cdot \sum_{m=1}^{k} J_k^{(m)}(F), \end{align*}
provided that $I_k(F) \neq 0$ and $J_k^{(m)}(F) \neq 0$ for each $m$, where
\begin{align*} I_k(F) :&= \idotsint_{[0,1]^{k}} F(t_1, \dots, t_k)^2\, \diff t_1 \diff t_2 \cdots \diff t_k, \\
	J_k^{(m)}(F):&= \idotsint_{[0,1]^{k-1}} \left(\int_{0}^{1} F(t_1, \dots, t_k)\, \diff t_m\right)^2 \diff t_1 \cdots  \diff t_{m-1} \diff t_{m+1} \cdots \diff t_k.
\end{align*}
\end{prop}

So with the weights $w(n)$ chosen as in Proposition \ref{prop:main-maynard}, we see that the ratio $S_2/S_1 \to \theta \frac{\sum_{m=1}^{k} J_k^{(m)}(F)}{I_k(F)}$ as $N\to\infty$. Following Maynard, we define \[ M_k\coloneqq\sup_{F}\frac{\sum_{m=1}^{k} J_k^{(m)}(F)}{I_k(F)}, \] where $F$ ranges over the functions satisfying the conditions of Proposition \ref{prop:main-maynard}. Since $\theta$ can be taken arbitrarily close to $\frac{1}{4}$, the ratio $S_2/S_1$ can be brought arbitrarily close to $\frac{1}{4}M_k$ by a judicious choice of $F$. Consequently, if $h_1, \dots, h_k$ is an admissible $k$-tuple, then the list $n+h_1, \dots, n+h_k$ contains at least $\lceil \frac{1}{4}M_k\rceil$ primes, for infinitely many values of $n$. To judge the strength of this last result, one needs estimates for $M_k$. The following lower bound on $M_k$ appears as \cite[Proposition 4.3]{maynard14}.

\begin{prop}  For all sufficiently large values of $k$,
	\[ M_k > \log{k}-2\log_2{k}-2.\]
\end{prop}

In fact, the Polymath8b project has strengthened this last result to $M_k > \log{k}+O(1)$. For our purposes, any result guaranteeing that $M_k\to\infty$ as $k\to\infty$ would be sufficient.

\subsection{A variant of the Maynard--Tao theorem} What has been said so far is enough to prove Theorem \ref{thm:main} for $\phi$ and $\sigma$. To  handle $\omega$ and $\tau$, we will use a variant of Proposition \ref{prop:main-maynard} allowing us to extend the range of pre--sieving to all primes $p \leq \log{N}/(\log\log{N})^2$, with at most one exception.

A careful statement of this result requires some preparation. By the Landau--Page theorem (see \cite[p.95]{davenport00}), for large $N$ there is at most one character $\chi_1$ of conductor $q_1 \leq Z\coloneqq \exp(5 \log{N}/(\log_2{N})^2)$ for which $L(s,\chi_1)$ has a zero in the interval $[1-\frac{3}{\log{Z}},1]$. In this case, $\chi_1$ is real, $q_1 > (\log{Z})^2$, and $q_1$ is a product of distinct odd primes together with either $1, 2, 4$, or $8$. If any such $q_1$ exists, we let $p_{\bad}$ denote its largest prime factor. Since the squarefull part of $q_1$ is absolutely bounded, $p_{\bad} \gg \log{q_1} \gg \log_2{N}$.

We will sketch a proof of the following result. In order to handle all cases at once, we let $p_{\bad}\coloneqq 1$ if the exceptional modulus $q_1$ does not exist.

\begin{prop}\label{prop:mm2} Make the same assumptions on $\theta$ and $F$ as in Proposition \ref{prop:main-maynard}. Redefine \[	W\coloneqq \prod_{\substack{p \leq \frac{\log{N}}{(\log_{2}{N})^2} \\ p \neq p_{\bad}}}p.
\]
Define $\lambda_{d_1,\dots, d_k}$ by \eqref{eq:lambdadef} if $(d_1\dots d_k, W p_{\bad})=1$, and put $\lambda_{d_1,\dots,d_k}=0$ otherwise. Define the weights $w(n)$ as in \eqref{eq:wndef}. Then the estimates of Proposition \ref{prop:main-maynard} for $S_1$ and $S_2$ still hold.
\end{prop}

Beyond Maynard's arguments \cite{maynard14} for Proposition \ref{prop:main-maynard}, the proof of Proposition \ref{prop:mm2} needs a version of the Bombieri--Vinogradov theorem due to Goldston--Pintz--Y{\i}ld{\i}r{\i}m (see \cite[Lemma 2]{GPY06}, and compare with \cite[Theorem 6]{GPY10}).

We adopt the notation $X_N \coloneqq \sum_{N \leq p < 2N} \chi_{\Pp}(n)$, and we let
\[ E(N;q) \coloneqq 1 + \max_{\substack{a\bmod{q} \\ (a,q)=1}} \bigg|\sum_{\substack{N \leq n < 2N \\ n\equiv a\pmod{q}}} \chi_{\Pp}(n) - \frac{X_N}{\phi(q)}\bigg|. \]
The following is a special case of the Goldston--Pintz--Y{\i}ld{\i}r{\i}m result.
\begin{prop}\label{prop:BV} Let $P$ be a positive integer with $P \leq \exp(2\log{N}/(\log_2{N})^2)$. Assume also that $p_{\bad}$ is coprime to $P$. Then
\[ \sum_{\substack{d \leq N^{1/2} \exp(-10\frac{\log{N}}{(\log_2{N})^2}) \\ (d,P \cdot p_{\bad})=1}} E(N; dP) \leq c_1 \frac{N}{P} \exp(-c_2 (\log_2{N})^2). \]
Here $c_1$ and $c_2$ are absolute positive constants.
\end{prop}

\begin{proof}[Sketch of the proof of Proposition \ref{prop:mm2}] We use a nearly identical argument to that of \cite{maynard14}. The key difference occurs in the proof of \cite[Lemma 5.2]{maynard14}, where the Bombieri--Vinogradov theorem is used to bound
\begin{equation}\label{eq:needbound} \sum_{\substack{d_1, \dots, d_k \\ e_1, \dots, e_k \\ [d_i,e_i]~\text{pairwise coprime}}} |\lambda_{d_1,\dots, d_k}| |\lambda_{e_1, \dots, e_k}| \cdot E(N; W \prod_{i=1}^{k} [d_i, e_i]). \end{equation}
To continue with the argument, one requires an upper bound for this expression of size $o(\frac{\phi(W)^k}{W^{k+1}} N (\log{N})^k)$. The argument of \cite{maynard14} gives a  bound of $O(N/(\log{N})^A)$, for any fixed $A$. This is sufficient in \cite{maynard14}, since in that context $W=(\log{N})^{o(1)}$. However, our $W$ grows faster than any fixed power of $\log{N}$, which requires us to invoke Proposition \ref{prop:BV} in place of Bombieri--Vinogradov.

To estimate \eqref{eq:needbound}, we group the tuples of $d_i$ and $e_i$ according to the value of $r \coloneqq \prod_{i=1}^{k}[d_i, e_i]$. By the choice of sieve weights, $\lambda_{d_1,\dots,d_k}$ vanishes unless $d_1\cdots d_k$ is squarefree, smaller than $R$, and coprime to $W p_{\bad}$. Thus, we can assume $r < R^2$, that $r$ is squarefree, and that $\gcd(r,W p_{\bad})=1$. Since there are $\tau_{3k}(r)$ tuples of $d_i$ and $e_i$ with $\prod_{i=1}^{k}[d_i,e_i]=r$, we find that \eqref{eq:needbound} is
\[ \ll \lambda_{\max}^2 \sum_{\substack{r < R^2 \\ (r,W p_{\bad})=1}} \mu^2(r) \tau_{3k}(r) E(N;rW), \]
where $\lambda_{\max} \coloneqq \max_{d_1,\dots, d_k}=|\lambda_{d_1,\dots,d_k}|$.
From \cite[eqs. (5.9) and (6.3)]{maynard14}, $\lambda_{\max} \ll_{k,F} (\log{R})^k$. Invoking the trivial bound $E(N;rW) \ll N/\phi(rW)$ and applying Cauchy--Schwarz gives that \eqref{eq:needbound} is
\begin{multline*} \ll (\log{R})^{2k} \bigg(\frac{N}{\phi(W)}\sum_{\substack{r < R^2 \\ (r,W p_{\bad})=1}} \frac{\mu^2(r) \tau_{3k}(r)^2}{\phi(r)}\bigg)^{1/2} \bigg(\sum_{\substack{r < R^2 \\ (r,W p_{\bad})=1}} E(N,rW)\bigg)^{1/2}  \\
\ll \frac{N}{W} \exp(-c_3 (\log\log{N})^2) \end{multline*}
for a certain positive constant $c_3$. Here we applied Proposition \ref{prop:BV} with $P=W$ to estimate the sum of $E(N;rW)$, using that $W \leq \exp((1+o(1)) \log{N}/(\log_2{N})^2)$, that $\gcd(p_{\bad},W)=1$, and that $R^2 = N^{2\theta}$ with $2\theta < \frac{1}{2}$.

The remainder of the proof of Proposition \ref{prop:mm2} continues along the path laid out in \cite{maynard14}. Following the proofs of \cite[Lemmas 6.2, 6.3]{maynard14}, the estimates of $S_1$ and $S_2$ now pick up factors of size $(p_{\bad}/\phi(p_{\bad}))^{O_k(1)}$, owing to the slightly different definition of the sieve weights. However, since $p_{\bad}=1$ or $p_{\bad} \gg \log_2{N}$, these factors are of the shape $1+o(1)$, and so the asymptotic estimates of Proposition \ref{prop:main-maynard} remain intact.
\end{proof}
\subsection{Auxiliary estimates} We recall briefly how the sum $S_1$ is handled. First, one unfolds the definition of $w(n)$ and reverses the order of summation to obtain that \[ S_1 = \sum_{\substack{d_1, \dots, d_k \\ e_1, \dots, e_k}} \lambda_{d_1, \dots, d_k} \lambda_{e_1, \dots, e_k} \sum_{\substack{N \leq n < 2N \\ n \equiv \nu\pmod{W} \\ [d_i,e_i]\mid n+h_i~\forall i}} 1. \] Now $\lambda_{d_1,\dots, d_k}$ vanishes unless $(d_1 \cdots d_k,W)=1$. It follows that the tuples $(d_1,\dots, d_k)$ and $(e_1,\dots,e_k)$ make a vanishing contribution to the double sum unless $[d_1,e_1], \dots, [d_k, e_k]$, and $W$ are pairwise relatively prime.  In the event that this latter condition holds, the inner sum is $\frac{N}{W\prod_{i=1}^{k}[d_i,e_i]}+O(1)$. We thus find that $S_1$ is, up to easy to estimate error terms, given by
\[ \tilde{S}_1\coloneqq \frac{N}{W} \sideset{}{'}\sum_{\substack{d_1, \dots, d_k\\ e_1, \dots, e_k}}\frac{\lambda_{d_1, \dots, d_k} \lambda_{e_1, \dots, e_k}}{\prod_{i=1}^{k}[d_i, e_i]}, \]
where the $'$ on the sum indicates the condition that $[d_1,e_1], \dots, [d_k, e_k], W$ are pairwise coprime. 
In the proofs of the estimate for $S_1$ in Proposition \ref{prop:main-maynard} and Proposition \ref{prop:mm2}, what is actually shown is that
\[ \tilde{S}_1 \sim \frac{\phi(W)^k}{W^{k+1}} N (\log{R})^k I_k(F). \]
We will need to estimate variants of the sum $\tilde{S}_1$ where extra coprimality conditions are imposed on the $d_i$ and $e_i$.
\begin{lem}\label{lem:tilde1} Let $q > 1$ be an integer coprime to $W$. Then
\begin{equation}\label{eq:l23lhs}\frac{N}{W} \sum_{\substack{d_1,\dots, d_k,  e_1, \dots, e_k \\ W,~[d_i, e_i],~q \text{ pairwise coprime}}}\frac{\lambda_{d_1, \dots, d_k} \lambda_{e_1, \dots, e_k}}{\prod_{i=1}^{k}[d_i, e_i]} = \tilde{S}_1 + O\left(\frac{1}{p_{\min}(q)} \cdot \frac{\phi(W)^k}{W^{k+1}} N (\log{N})^{k}\right). \end{equation}
Here $p_{\min}(q)$ denotes the smallest prime factor of $q$, and the implied constant may depend on $k$, $F$, and the number of distinct primes dividing $q$.
\end{lem}

\begin{remark} In particular, the left-hand side of \eqref{eq:l23lhs} is $O(\frac{\phi(W)^k}{W^{k+1}} N(\log{N})^k)$.
\end{remark}

\begin{proof} We will assume throughout that $q$ is squarefree; otherwise, we may replace $q$ by the product of its distinct prime factors. To estimate the left-hand side of \eqref{eq:l23lhs}, we borrow heavily from Maynard's proof of \cite[Lemma 5.1]{maynard14}. Recall that $\lambda_{d_1, \dots, d_k}$ vanishes unless $d_1\cdots d_k$ is squarefree and coprime to $W$. Thus, we may drop the requirement that each $d_i$ is prime to $W$ and the requirement that the $d_i$ are pairwise coprime, since the exceptional terms make no contribution. The same comment applies, of course, to the $e_i$. Thus, the only surviving conditions on the sum are that each $(d_i,e_j)=1$ and that each $d_i$ and $e_j$ is prime to $q$.  We incorporate the restrictions that $(d_i, e_j)=1$ by multiplying through by $\sum_{s_{i,j} \mid d_i, e_j} \mu(s_{i,j})$ for $i \neq j$. Similarly, we incorporate the restrictions that $(d_i,q)=(e_j,q)=1$ by multiplying through by $\sum_{\delta_i \mid d_i, q} \mu(\delta_i)$ and $\sum_{\epsilon_j \mid e_j, q} \mu(\epsilon_j)$, for all pairs of $i$ and $j$. Finally, we rewrite
\[ \frac{1}{[d_i, e_i]} = \frac{1}{d_i e_i} \sum_{u_i \mid d_i, e_i} \phi(u_i). \]
These transformations put the left-hand side of \eqref{eq:l23lhs} in the form
\begin{multline*} \frac{N}{W} \sum_{u_1, \dots, u_k} \left(\prod_{i=1}^{k} \phi(u_i)\right) \sum_{s_{1,2}, \dots, s_{k,k-1}} \left(\prod_{\substack{1 \leq i, j \leq k \\ i \neq j}} \mu(s_{i,j})\right) \sum_{\substack{\delta_1, \dots, \delta_k \mid q\\ \epsilon_1, \dots, \epsilon_k \mid q}}\left(\prod_{i=1}^{k} \mu(\delta_i) \prod_{j=1}^{k} \mu(\epsilon_j)\right) \\
\times \sum_{\substack{d_1,\dots, d_k \\ e_1, \dots, e_k \\ u_i \mid d_i, e_i\,\forall i \\ s_{i,j}\mid d_i, e_j\,\forall i\neq j\\ \delta_i \mid d_i, \epsilon_j \mid e_j\, \forall i, j}} \frac{\lambda_{d_1,\dots, d_k} \lambda_{e_1, \dots, e_k}}{\prod_{i=1}^{k} d_i e_i}.
\end{multline*}
To proceed, we follow Maynard in making the change of variables
\[ y_{r_1, \dots, r_k} = \left(\prod_{i=1}^{k} \mu(r_i) \phi(r_i)\right) \sum_{\substack{d_1, \dots, d_k \\ r_i \mid d_i~\forall i}} \frac{\lambda_{d_1, \dots, d_k}}{\prod_{i=1}^{k} d_i}. \]
Note that $y_{r_1, \dots, r_k}$ vanishes unless $r_1\cdots r_k$ is squarefree, prime to $W$, and smaller than $R = N^{\theta}$. This change of variables allows us to rewrite the sum in the form
\begin{multline}\label{eq:tilde1rewritten} \frac{N}{W} \sum_{u_1, \dots, u_k} \left(\prod_{i=1}^{k} \phi(u_i)\right) \sideset{}{^*}\sum_{s_{1,2}, \dots, s_{k, k-1}} \left(\prod_{\substack{1 \leq i, j \leq k \\ i \neq j}} \mu(s_{i,j})\right) \sum_{\substack{\delta_1, \dots, \delta_k \mid q \\ \epsilon_1, \dots, \epsilon_k \mid q}} \left(\prod_{i=1}^{k} \mu(\delta_i) \prod_{j=1}^{k}\mu(\epsilon_j)\right) \\ \times \Bigg(\prod_{i=1}^{k} \frac{\mu(a_i)}{\phi(a_i)}\Bigg) \Bigg(\prod_{j=1}^{k} \frac{\mu(b_j)}{\phi(b_j)}\Bigg) y_{a_1, \dots, a_k} y_{b_1, \dots, b_k}, \end{multline}
where \[ a_i  \coloneqq \lcm[u_i \prod_{j \neq i} s_{i,j}, \delta_i] \quad\text{and}\quad b_j \coloneqq \lcm[u_j \prod_{i \neq j} s_{i,j}, \epsilon_j].\] Here the $*$ on the sum indicates that $s_{i,j}$ is restricted to be coprime to $u_i$, $u_j$, $s_{i,a}$, and $s_{b,j}$ for all $a \neq j$ and $b\neq i$; the other $s_{i,j}$ make no contribution. The contribution to \eqref{eq:tilde1rewritten} from the terms where $\delta_i = \epsilon_j=1$ for all $i, j$ is precisely equal to $\tilde{S}_1$. (To see this, think of transforming $\tilde{S}_1$ the same way we transformed the left-hand side of \eqref{eq:l23lhs}; alternatively, compare with eq. (5.10) of \cite{maynard14}.) The total number of possibilities for $\delta_i$ and $\epsilon_j$ is bounded in terms of the number of distinct prime factors of $q$. We complete the proof by showing that any choice of $\delta_i$ and $\epsilon_j$ having at least one of the variables larger than $1$ makes a contribution
\begin{equation}\label{eq:acceptablecontribution} \ll \frac{1}{p_{\min}(q)} \cdot \frac{\phi(W)^k}{W^{k+1}}\cdot N (\log{N})^{k}. \end{equation}
Consider any such choice of the $\delta$'s and $\epsilon$'s. For definiteness, we suppose that $\delta_1 > 1$; the other cases are entirely similar. For each term in \eqref{eq:tilde1rewritten} corresponding to this choice, we define $\delta_i'$ and $\epsilon_j'$ by the equations
\[ a_i = \Bigg(u_i \prod_{j \neq i} s_{i,j}\Bigg) \delta_i', \qquad b_j = \Bigg(u_j \prod_{i \neq j} s_{i,j}\Bigg) \epsilon_j'. \]
Since $q$ is squarefree, we get that $\delta_i'$ is a unitary divisor of $a_i$ and $\epsilon_j'$ is a unitary divisor of $b_j$. Thus, we can write $\mu(a_i) = \left(\mu(u_i) \prod_{j \neq i} \mu(s_{i,j})\right) \mu(\delta_i')$, and similarly for $\mu(b_j)$, $\phi(a_i)$, and $\phi(b_j)$. Making these substitutions, we see that the contribution to \eqref{eq:tilde1rewritten} from our choice of $\delta$'s and $\epsilon$'s is given by
\begin{multline*} \frac{N}{W} \sum_{u_1, \dots, u_k} \left(\prod_{i=1}^{k} \frac{\mu(u_i)^2}{\phi(u_i)} \right) \sideset{}{^*}\sum_{s_{1,2}, \dots, s_{k, k-1}}\left(\prod_{\substack{1 \leq i, j \leq k\\i \neq j}} \frac{\mu(s_{i,j})}{\phi(s_{i,j})^2}\right) \\ \times \left(\prod_{i=1}^{k} \frac{\mu(\delta_i)\mu(\delta_i')}{\phi(\delta_i')} \prod_{j=1}^{k}\frac{\mu(\epsilon_j) \mu(\epsilon_j')}{\phi(\epsilon_j')}\right) y_{a_1, \dots, a_k} y_{b_1, \dots, b_k}. \end{multline*}
Now replace all terms by their absolute values. Whenever $\delta_1' > 1$, we have $\prod_{i}\frac{1}{\phi(\delta_i')} \prod_{j}\frac{1}{\phi(\epsilon_j')} \leq \frac{1}{p_{\min}(\delta_1')-1} \ll \frac{1}{p_{\min}(q)}$. So the terms with $\delta_1' > 1$ give us a contribution that is
\[ \ll \frac{y_{\max}^2}{p_{\min}(q)} \frac{N}{W} \Bigg(\sum_{\substack{u < R \\ (u,W)=1}}\frac{\mu(u)^2}{\phi(u)}\Bigg)^{k} \Bigg(\sum_{s} \frac{\mu(s)^2}{\phi(s)^2}\Bigg)^{k(k-1)} \ll \frac{y_{\max}^2}{p_{\min}(q)} \cdot \frac{\phi(W)^k}{W^{k+1}} \cdot N (\log{N})^{k},\]
where $y_{\max} = \max_{r_1, \dots, r_k} |y_{r_1, \dots, r_k}|$. Suppose instead that $\delta_1'=1$. Then $p_{\min}(\delta_1)$ divides either $u_1$ or some $s_{1,j}$ with $1 < j \leq k$. The terms corresponding to the former case contribute
\begin{multline*} \ll y_{\max}^2 \frac{N}{W}   \Bigg(\sum_{\substack{u < R/p_{\min}(\delta_1) \\ (u, p_{\min}(\delta_1) \cdot W)=1}}\frac{\mu(u)^2}{\phi(u) \phi(p_{\min}(\delta_1))}\Bigg) \Bigg(\sum_{\substack{u < R \\ (u,W)=1}}\frac{\mu(u)^2}{\phi(u)}\Bigg)^{k-1} \Bigg(\sum_{s} \frac{\mu(s)^2}{\phi(s)^2}\Bigg)^{k(k-1)} \\ \ll \frac{y_{\max}^2}{p_{\min}(q)} \cdot \frac{\phi(W)^k}{W^{k+1}} \cdot N (\log{N})^k;\end{multline*}
we used here the trivial bounds $\prod_{i} \frac{1}{\phi(\delta_i')} \prod_{j} \frac{1}{\phi(\epsilon_j')} \leq 1$ and $p_{\min}(\delta_1) \geq p_{\min}(q)$. Finally, the terms where $p_{\min}(\delta_1)$ divides $s_{1,j}$ for some $j$ contribute
\begin{multline*} \ll y_{\max}^2 \frac{N}{W}  \Bigg(\sum_{\substack{u < R \\ (u,W)=1}}\frac{\mu(u)^2}{\phi(u)}\Bigg)^{k} \Bigg(\sum_{s} \frac{\mu(p_{\min}(\delta_1) s)^2}{\phi(p_{\min}(\delta_1) s)^2}\Bigg) \Bigg(\sum_{s} \frac{\mu(s)^2}{\phi(s)^2}\Bigg)^{k(k-1)-1} \\ \ll \frac{y_{\max}^2}{p_{\min}(q)^2} \cdot \frac{\phi(W)^k}{W^{k+1}} \cdot N (\log{N})^k. \end{multline*}
We recall that the choice of $\lambda$'s was made so that whenever $y_{r_1,\dots, r_k}\neq 0$, we have $y_{r_1,\dots,r_k} = F(\frac{\log{r_1}}{\log{R}},\dots,\frac{\log{r_k}}{\log{R}})$; see \cite[eq. (6.3)]{maynard14}. In particular, $y_{\max}=O(1)$. Now combining the last three displayed estimates yields \eqref{eq:acceptablecontribution}.
\end{proof}

We will also need the following upper bound for $w(n)$ in mean square, which appears as \cite[Lemma 3.5]{pollack14}. While this lemma was originally proved in the context of Proposition \ref{prop:main-maynard}, the same argument goes through with the modified definitions of Proposition \ref{prop:mm2}.

\begin{lem}\label{lem:csprep} Assume $\theta$ and $F$ satisfy the conditions of Proposition \ref{prop:main-maynard}. We have
	\[ \sum_{\substack{N \leq n < 2N \\ n \equiv \nu\pmod{W}}} w(n)^2 \ll_{k, \theta, F} \frac{N}{W} (\log{R})^{19k}.  \]
\end{lem}

\section{Proof of Theorem \ref{thm:main}} While the proof of Theorem \ref{thm:main} takes different forms depending on the specific arithmetic function under consideration, certain parameters will remain fixed throughout. Given $K$, we fix $k \geq 2$ with $\lceil \frac{1}{4} M_k\rceil > K-1$. This choice of $k$ guarantees that there is a $\theta < \frac{1}{4}$ and a function $F$ so that the limit of $S_2/S_1$, as $N\to\infty$, exceeds $K-1$. We will keep $K$, $k$, $\theta$, and $F$ fixed throughout this section. Implied constants may depend on any of these parameters without further mention.

\subsection{The Euler $\phi$-function} In this section and the next, we will only be applying Proposition \ref{prop:main-maynard}, and so we assume that $W$ and $\lambda$'s are defined as in the statement of that result.

Our arguments use a probabilistic interpretation of Proposition \ref{prop:main-maynard}. Let $\Hh$ be a fixed admissible tuple. For each large $N$, we view $\Omega$ as a finite probability space where the probability mass at each $n_0$ is given by $w(n_0)/\sum_{n \in \Omega}w(n)$. Then $X\coloneqq\sum_{i=1}^{k} \chi_{\Pp}(n+h_i)$ can be viewed as a random variable on $\Omega$. (Of course, $\Omega$ and $X$ depend on $N$, but we suppress this in the notation.) Proposition \ref{prop:main-maynard} shows that $\E[X]\to S_2/S_1$, as $N\to\infty$. Since $X$ only assumes values in $\{0, 1, 2, \dots, k\}$, our choices of $K, k, \theta$, and $F$ give us that
	\begin{equation}\label{eq:proboflotsofprimes} \Prob(X\geq K) \geq \E[\frac{1}{k}(X-(K-1))] = \frac{1}{k}\left(\E[X]-(K-1)\right)\gg 1,\end{equation}
for large values of $N$. In other words, the probability that at least $K$ of $n+h_1, \dots, n+h_k$ are prime is bounded away from zero.

These results hold for any admissible tuple $\Hh$ and for any choice of $\nu\bmod{W}$ satisfying \eqref{eq:basicnu}. By tweaking these choices, we can bias the sieve to detect long tuples of consecutive primes $p$ on which $\phi(p-1)$ is monotone.

We treat the decreasing case first. In that case, it is convenient to work with the fixed tuple $\Hh$ specified by
\begin{equation}\label{eq:hidef} h_i = (i-1) (2k)!. \end{equation}
It is easy to see that this is admissible. With this choice of $\Hh$, we will describe how to pick the residue class $\nu\bmod{W}$ to satisfy, in addition to \eqref{eq:basicnu}, the following two constraints:
\begin{enumerate}
\item[(1)] For $n\in \Omega$, any prime between $n+h_1$ and $n+h_k$ is one of $n+h_1,\dots, n+h_k$.
\item[(2)] With probability $1+o(1)$ (as $N\to\infty$), we have $\frac{n+h_i-1}{\phi(n+h_i-1)} \in (2^{4i}, 2^{4i+3}]$ for all $1 \leq i \leq k$.
\end{enumerate}
From (1), (2), and \eqref{eq:proboflotsofprimes}, there is --- once $N$ is large enough --- a positive probability that an $n\in \Omega$ satisfies all of	
\begin{itemize}
\item at least $K$ of $n+h_1, \dots, n+h_k$ are prime,
\item the primes in the list $n+h_1, \dots, n+h_k$ exhaust the list of primes between $n+h_1$ and $n+h_k$,
\item each ratio $\frac{n+h_i-1}{\phi(n+h_i-1)} \in (2^{4i}, 2^{4i+3}]$.
\end{itemize}
Now the intervals $(2^{4i}, 2^{4i+3}]$ are disjoint since each is strictly to the right of the previous. Moreover, as $N\to\infty$, the numerators in the fractions $\frac{n+h_i-1}{\phi(n+h_i-1)}$ are all asymptotic to each other. So if $n \in \Omega$ has these three properties (and $N$ is large enough), then the set of primes $p$ between $n+h_1$ and $n+h_k$ (inclusive) has at least $K$ elements, and $\phi(p-1)$ is decreasing on this set.

It remains to choose $\nu$ and to prove that $\nu$ has the properties (1) and (2).

\subsubsection*{Choice of $\nu\bmod{W}$}
We pick $\nu\bmod{W}$ by choosing $\nu\bmod{p}$ for each prime $p \leq \log_3{N}$. We begin by selecting $\nu \equiv 1\pmod{2}$.
For an odd prime $p$, we say that a \emph{default selection of $\nu\bmod{p}$} is any $\nu$ with
\[ \nu \not\equiv h_1, \dots, h_k, 1-h_1, \dots, 1-h_k\pmod{p}. \]
There is always at least one way to make a default selection of $\nu$: For $p > 2k$, we can certainly avoid the at most $2k$ residue classes $h_1, \dots, h_k, 1-h_1, \dots, 1-h_k \bmod{p}$. If $p \leq 2k$, then $p$ divides every $h_i$, and we only have to avoid the two residue classes $0$ and $1$ modulo $p$.

We choose $\nu$ mod $p$ in different ways depending on the size of the odd prime $p$:
\begin{enumerate}
	\item[Range I.] $p \leq \log_4{N}$: For each $p$ in this range, make a default selection of $\nu$ for $p$.
	\item[Range II.] $\log_4{N} < p \leq \frac{1}{2}\log_3{N}$: Use the greedy algorithm to choose disjoint sets of primes $\Pp_i$, $1 \leq i \leq k$, satisfying
	\[ 4i \log{2} \leq \log{2} + \sum_{p \in \Pp_i}\log \frac{p}{\phi(p)} \leq (4i+1) \log{2}. \]
	For $p \in \Pp_i$, choose $\nu \equiv 1-h_i\pmod{p}$. For $p$ in this range not belonging to any of the $\Pp_i$, make a default selection.
	\item[Range III.] $\frac{1}{2}\log_3{N} < p \leq \log_3{N}$: For each even number $h\not\in\Hh$ that lies between $h_1$ and $h_k$, we choose a distinct prime $p^{(h)}$ from this range. We then choose $\nu \equiv -h\pmod{p^{(h)}}$ for each of these $h$. For the remaining values of $p$, we make a default selection of $\nu$.
\end{enumerate}

We leave to the reader the straightforward verification that the necessary condition \eqref{eq:basicnu} is satisfied for this choice of $\nu$.

\subsubsection*{Verification of (1) and (2)} From our selections in Range III, if $h$ is between $h_1$ and $h_k$ but not in $\Hh$, then $n+h$ has a known prime divisor --- either $p=2$ if $h$ is odd or the prime $p^{(h)} \le \log_3{N}$ if $h$ is even. So $n+h$ is necessarily composite. Thus, condition (1) above is satisfied with this choice of $\nu$.

We now turn to condition (2). Fix $1 \leq i \leq k$. Suppose $n\in\Omega$, i.e., $n \equiv \nu\pmod{W}$ with $N \leq n < 2N$. Our selections in Ranges I and II show that the primes not exceeding $\log_3{N}$ that divide $n+h_i-1$ are precisely $2$ and the primes in $\Pp_i$. So by our choice of $\Pp_i$,
\[ 2^{4i} \leq \frac{n+h_i-1}{\phi(n+h_i-1)} \leq 2^{4i+1} \prod_{\substack{p \mid n+h_i-1 \\ \log_3{N} < p \leq \log{N}}} \frac{p}{\phi(p)} \prod_{\substack{p \mid n+h_i-1 \\ p > \log{N}}} \frac{p}{\phi(p)}. \]
For the second product, we have the easy estimate
\begin{align*} \prod_{\substack{p \mid n+h_i-1 \\ p > \log{N}}} \frac{p}{\phi(p)} &\leq \left(1 + \frac{1}{\log{N}-1}\right)^{\log(3N)/\log_2{N}} \\&\leq 1 + O(1/\log_2{N}), \end{align*}
using that the integer $n+h_i-1\leq 3N$ has at most $\log(3N)/\log_2{N}$ prime factors exceeding $\log{N}$. Thus, once $N$ is large,
\[ 2^{4i} \leq \frac{n+h_i-1}{\phi(n+h_i-1)} \leq 2^{4i+2} \prod_{\substack{p \mid n+h_i-1 \\ \log_3{N} < p \leq \log{N}}} \frac{p}{\phi(p)}. \]
So to obtain condition (2) above, it is enough to show that the remaining product is bounded by $2$ with probability $1+o(1)$, as $N\to\infty$.  This follows immediately from the next lemma and Markov's inequality in elementary probability theory.

\begin{lem}\label{lem:phipexpectation}  Let $1 \le i \le k$. Then
\[ \E\bigg[\sum_{\substack{p \mid n+h_i-1 \\ \log_3{N} < p \leq \log{N}}} \log \frac{p}{\phi(p)}\bigg] = o(1),\] as $N\to\infty$.
\end{lem}
\begin{proof} Since $\log \frac{p}{\phi(p)} \ll \frac{1}{p}$, we see that
\begin{multline*} \E\bigg[\sum_{\substack{p \mid n+h_i-1 \\ \log_3{N} < p \leq \log{N}}} \log \frac{p}{\phi(p)}\bigg] = S_1^{-1} \sum_{\log_3{N} < p \leq \log{N}} \log\frac{p}{\phi(p)} \sum_{\substack{N \leq n < 2N \\ n \equiv \nu\pmod{W} \\ p \mid n+h_i-1}} w(n) \\ \ll S_1^{-1} \sum_{\log_3{N} < p \leq \log{N}} \frac{1}{p}\bigg(\sum_{\substack{d_1, \dots, d_k \\ e_1, \dots, e_k}} \lambda_{d_1, \dots, d_k} \lambda_{e_1, \dots, e_k} \sum_{\substack{N \leq n < 2N \\ n \equiv \nu\pmod{W} \\ [d_i, e_i] \mid n+h_i~\forall i\\p \mid n+h_i-1}} 1\bigg). \end{multline*}
The tuples $(d_1, \dots, d_k)$ and $(e_1,\dots, e_k)$ make a contribution to the inner double sum only if $W, [d_1, e_1], \dots, [d_k,e_k]$, and $p$ are relatively prime, in which case the innermost sum has size $\frac{N}{p W \prod_{i=1}^{k} [d_i,e_i]} + O(1)$. Using $''$ to denote this restriction on the $d_i$ and $e_i$, we find that
\begin{multline}\label{eq:expcomp}\sum_{\log_3{N} < p \leq \log{N}} \frac{1}{p}\bigg(\sum_{\substack{d_1, \dots, d_k \\ e_1, \dots, e_k}} \lambda_{d_1, \dots, d_k} \lambda_{e_1, \dots, e_k} \sum_{\substack{N \leq n < 2N \\ n \equiv \nu\pmod{W} \\ [d_i, e_i] \mid n+h_i~\forall i\\p \mid n+h_i-1}} 1\bigg) \\= \sum_{\log_3{N} < p \leq \log{N}}\frac{1}{p^2} \left(\frac{N}{W}\sideset{}{''}\sum_{\substack{d_1, \dots, d_k \\ e_1, \dots, e_k}} \frac{\lambda_{d_1, \dots, d_k} \lambda_{e_1, \dots, e_k}}{\prod_{i=1}^{k}[d_i, e_i]}\right) \\ + O\left(\left (\sum_{\log_3{N} < p \leq \log{N}} \frac{1}{p}\right) \left(\sum_{d_1, \dots, d_k} |\lambda_{d_1, \dots, d_k}|\right)^2\right). \end{multline}
The $O$-term here is
\[ \ll \log_3{N} \cdot \lambda_{\max}^2 \left(\sum_{r < R} \tau_k(r)\right)^2 \ll R^2 (\log{R})^{4k-2}\cdot \log_3{N}, \]
recalling that $\lambda_{\max} \ll (\log{R})^k$. Since $R = N^{\theta}$ with $\theta < \frac{1}{4}$, the $O$-term is $o(N^{1/2})$, which is tiny compared to $\frac{\phi(W)^k}{W^{k+1}} N (\log{N})^k$. From the remark following Lemma \ref{lem:tilde1}, the main term in \eqref{eq:expcomp} is
\[ \ll \bigg(\sum_{\log_3{N} < p \leq \log{N}}\frac{1}{p^2}\bigg) \cdot \bigg(\frac{\phi(W)^k}{W^{k+1}} N(\log{N})^k\bigg) \ll S_1/\log_3{N}. \]
Collecting our estimates gives $\E\bigg[\sum_{\substack{p \mid n+h_i-1 \\ \log_3{N} < p \leq \log{N}}} \log \frac{p}{\phi(p)}\bigg] = O(1/\log_3{N})$.
\end{proof}

This completes the proof of Theorem \ref{thm:main} for $\phi$, in the decreasing case. The increasing case follows in precisely the same way, but with the tuple $\Hh$ now given by $h_i = -(i-1)(2k)!$ for $1 \leq i \leq k$.

\subsection{The sum-of-divisors function $\sigma$} The proof that there are long increasing runs of $\sigma$ along consecutive shifted primes is parallel to the argument we just gave for long decreasing runs for $\phi$, once one replaces each ratio $\frac{n+h_i-1}{\phi(n+h_i-1)}$ with $\frac{\sigma(n+h_i-1)}{n+h_i-1}$. We quickly sketch the other necessary changes.

We once again assume $\Hh$ is defined by $h_i = (i-1)(2k)!$ for $1 \leq i \leq k$. The only difference in the choice of $\nu\bmod{W}$ is that in the range $\log_4{N} < p \leq \frac{1}{2}\log_3{N}$, we select the disjoint sets $\Pp_i$ so that
\begin{equation}\label{eq:selectionguarantee} 4i\log{2} \leq \log\frac{3}{2} + \sum_{p \in \Pp_i} \log \frac{\sigma(p)}{p} \leq (4i+1)\log{2}.\end{equation}
Exactly as before, the choice of $\nu$ immediately gives us
\begin{enumerate}
\item[(1$'$)] For $n\in \Omega$, any prime between $n+h_1$ and $n+h_k$ is one of $n+h_1,\dots, n+h_k$.
\end{enumerate}
It remains to verify that we also have
\begin{enumerate}
\item[(2$'$)] Each $\frac{\sigma(n+h_i-1)}{n+h_i-1} \in (2^{4i}, 2^{4i+3}]$ with probability $1+o(1)$.
\end{enumerate}
Here a small additional argument is required to account for the fact that $\frac{\sigma(n+h_i-1)}{n+h_i-1}$ is sensitive not only to the primes dividing $n+h_i-1$ but also to their multiplicities.

Fix $1 \leq i \leq k$. For any $N \leq n < 2N$ with $n\equiv \nu\pmod{W}$, \eqref{eq:selectionguarantee} guarantees the lower bound \[ \frac{\sigma(n+h_i-1)}{n+h_i-1} \ge 2^{4i}. \] In addition, we have
\begin{equation}\label{eq:squarefreeupper} \prod_{p \in \Pp_i \cup\{2\}}\frac{\sigma(p)}{p} \le 2^{4i+1}.\end{equation}
To go from \eqref{eq:squarefreeupper} to an upper bound on $\frac{\sigma(n+h_i-1)}{n+h_i-1}$, we use the following estimates, valid for every prime power $p^e$:
\[ \frac{\sigma(p^e)}{p^e} = 1+ \frac{1}{p} + \dots + \frac{1}{p^e} < \frac{p}{p-1}=\frac{\sigma(p)}{p} \frac{p^2}{p^2-1}. \]
Since the only primes $p\le \log_3{N}$ dividing $n+h_i-1$ are $2$ and the primes in $\Pp_i$,
\begin{align*} \frac{\sigma(n+h_i-1)}{n+h_i-1} &\leq\left(\prod_{p \in \Pp_i \cup\{2\}} \frac{\sigma(p)}{p} \prod_{p \in \Pp_i  \cup \{2\}} \frac{p^2}{p^2-1}\right) \prod_{\substack{p \mid n+h_i-1 \\ \log_3{N} < p \leq \log{N}}} \frac{p}{p-1}\prod_{\substack{p \mid n+h_i-1 \\ p > \log{N}}} \frac{p}{p-1} \\
	&\le 2^{4i+1} \cdot \left(\frac{4}{3}\prod_{p \in \Pp_i} \frac{p^2}{p^2-1}\right) \prod_{\substack{p \mid n+h_i-1 \\ \log_3{N} < p \leq \log{N}}} \frac{p}{p-1}\prod_{\substack{p \mid n+h_i-1 \\ p > \log{N}}} \frac{p}{p-1}.\end{align*}
The product over $p > \log{N}$ was estimated in the preceding section, where it was shown to be $1+o(1)$. To estimate the product over $p \in \Pp_i$, note that every element of $\Pp_i$ exceeds $\log_4{N}$ while each term in the product has the shape $1+O(1/p^2)$. Consequently, that product is also $1+o(1)$. Since $\frac{4}{3} < 2$, we conclude that for large $N$,
\[ 2^{4i} \le \frac{\sigma(n+h_i-1)}{n+h_i-1} \le 2^{4i+2} \prod_{\substack{p \mid n+h_i-1 \\ \log_3{N} < p \leq \log{N}}} \frac{p}{p-1}.\]
Lemma \ref{lem:phipexpectation} and Markov's bound now yield (2$'$). The rest of the proof of the increasing case for $\sigma$ is completed exactly as in the proof of the decreasing case for $\phi$. Long decreasing runs are obtained by replacing $\Hh$ with $-\Hh$.

\subsection{The count of prime divisors} For the remainder of the proof of Theorem \ref{thm:main}, we switch from using Proposition \ref{prop:main-maynard} to using Proposition \ref{prop:mm2}. We continue to view $\Omega$ as a finite probability space where the mass at $n_0$ is given by $w(n_0)/\sum_{n \in \Omega}w(n)$.

We again choose $\Hh$ as the admissible tuple defined by \eqref{eq:hidef}.

We now select $\nu\bmod{W}$. We choose $\nu \equiv 1\pmod{2}$, and for odd primes $p$, we proceed as follows:
\begin{enumerate}
\item[Range I.] $p \leq \frac{1}{4} \frac{\log{N}}{(\log_2{N})^2}$: Make a default selection of $\nu$ for $p$. In other words, choose $\nu$ so that $\nu \not\equiv h_1, \dots, h_k, 1-h_1, \dots, 1-h_k\pmod{p}$.
\item[Range II.] $\frac{1}{4} \frac{\log{N}}{(\log_2{N})^2} < p \leq \frac{1}{2} \frac{\log{N}}{(\log_2{N})^2}$: Choose disjoint sets of primes $\Pp_i$, $1 \leq i \leq k$, satisfying $\#\Pp_i = \lfloor i (\log_2{N}) (\log_3{N})\rfloor$.	For $p \in \Pp_i$, choose $\nu \equiv 1-h_i\pmod{p}$. For $p$ in this range not belonging to any of the $\Pp_i$, make a default selection of $\nu$ for $p$.
	\item[Range III.] $\frac{1}{2} \frac{\log{N}}{(\log_2{N})^2} < p \leq \frac{\log{N}}{(\log_2{N})^2}$: For each even $h\not\in\Hh$ that lies between $h_1$ and $h_k$, we choose a distinct prime $p^{(h)}$ from this range. We then choose $\nu \equiv -h\pmod{p^{(h)}}$ for each of these $h$. For the remaining values of $p$, we make a default selection of $\nu$ for $p$.
\end{enumerate}
From our initial choices of $k$, $F$, and $\theta$, there is (as before) a positive probability that an $n \in \Omega$ satisfies
\begin{enumerate}
\item[(0)] at least $K$ of $n+h_1, \dots, n+h_k$ are prime.
\end{enumerate}
Because of our choice of $\nu$ in Range III, each such  $n$ has the property that
\begin{enumerate}
\item[(1)] the only primes between $n+h_1$ and $n+h_k$ are on the list $n+h_1, \dots, n+h_k$.
\end{enumerate}
We now argue that for $n\in \Omega$,
\begin{enumerate}
\item[(2)] with probability $1+o(1)$, as $N\to\infty$, we have
\[ 0 < \omega(n+h_i-1) - i (\log_2{N}) (\log_3{N}) < (\log_2{N}) (\log_4{N}) \]
for all $1 \leq i \leq k$.
\end{enumerate}
Clearly, (1), (2), and our choice of $\Hh$ imply that with positive probability, $\omega(p-1)$ is strictly increasing on the set of primes $p$ between $n+h_1$ and $n+h_k$, while condition (0) tells us that there are at least $K$ such primes.

Let us tally the primes dividing $n+h_i-1$. Those not exceeding $\log{N}/(\log_2{N})^2$ are $2$, the primes in $\Pp_i$ --- of which there are $\lfloor i (\log_2{N})(\log_3{N}) \rfloor$ --- and possibly $p_{\bad}$. Notice that the number of prime divisors of $n+h_i-1$ exceeding $N^{1/\log_2{N}}$ is $O(\log_2{N})$. So to complete the proof of (2), it is enough to show that the number of primes $p$ dividing $n+h_i-1$ with $p \in \Ii\coloneqq (\log{N}/(\log_2{N})^2, N^{1/\log_2{N}}]$ is smaller than $\frac{1}{2}(\log_2{N}) (\log_4{N})$ with probability $1+o(1)$. Let $\tilde{\omega}(n) \coloneqq \sum_{p \mid n,~p \in \Ii} 1$. The desired estimate follows immediately from the next lemma.

\begin{lem} Let $1 \leq i \leq k$. As $N\to\infty$,
	\[ \E[\tilde{\omega}(n+h_i-1)] \sim \log_2{N}. \]
\end{lem}
\begin{proof} We have $\E[\tilde{\omega}(n+h_i-1)] = S_1^{-1}  \sum_{n \in \Omega} \tilde{\omega}(n+h_i-1) w(n)$. Expanding out the definitions of $\tilde{\omega}$ and of $w(n)$, the sum on $n$ becomes
	\[ \sum_{p \in \Ii} \Bigg(\sum_{\substack{d_1, \dots, d_k \\ e_1, \dots, e_k}} \lambda_{d_1, \dots, d_k} \lambda_{e_1 \dots, e_k} \sum_{\substack{N \leq n < 2N \\ n \equiv \nu\pmod{W} \\ [d_i,e_i]\mid n+h_i~\forall i\\ p \mid n+h_i-1}} 1\Bigg). \]
The only way the tuples $(d_1, \dots, d_k)$ and $(e_1, \dots, e_k)$ contribute to the inner double sum is if $[d_1, e_1],\dots, [d_k, e_k], W$, and $p$ are pairwise coprime; in that case, the innermost sum is $\frac{N}{pW \prod_{i=1}^{k}[d_i,e_i]} + O(1)$. This gives rise to an error term that is
\begin{align*} \ll \left(\sum_{p \in \Ii} 1\right) \left(\sum_{d_1, \dots, d_k} |\lambda_{d_1, \dots, d_k}|\right)^2 &\ll N^{1/\log_2{N}} \cdot \lambda_{\max}^2 \left(\sum_{r < R} \tau_k(R)\right)^2 \\
&\ll N^{1/\log_2{N}} \cdot R^2 (\log{R})^{4k}, \end{align*}
which is smaller than $N^{1/2}$ for large $N$. (We used again that $\lambda_{\max} \ll (\log{R})^k$.) In particular, the error here is $o(\frac{\phi(W)^k}{W^{k+1}} N(\log{N})^k)$, since $W = N^{o(1)}$. Using Lemma \ref{lem:tilde1}, the main term contributes
\begin{align*} \sum_{p \in \Ii} \frac{N}{pW} &\sum_{\substack{d_1,\dots, d_k,  e_1, \dots, e_k \\ W,~[d_i, e_i],~p \text{ pairwise coprime}}}\frac{\lambda_{d_1, \dots, d_k} \lambda_{e_1, \dots, e_k}}{\prod_{i=1}^{k}[d_i, e_i]} =
\sum_{p \in \Ii} \frac{1}{p} \bigg(\tilde{S}_1 + O\big(\frac{1}{p} \frac{\phi(W)^k}{W^{k+1}} N (\log{N})^k\big)\bigg)\\
&= \tilde{S_1} (\log_2{N} + O(\log_3{N})) + o\big(\frac{\phi(W)^k}{W^{k+1}} N (\log{N})^k\big).
\end{align*}
Since $S_1 \sim \tilde{S}_1$ and both of these quantities have order $\frac{\phi(W)^k}{W^{k+1}} N (\log{N})^k$, the result follows.
\end{proof}

This completes the proof of the increasing case; the decreasing case follows from the same arguments upon replacing $\Hh$ with $-\Hh$.

\subsection{The count-of-divisors function}  Finally, we tackle the  function $\tau(n)$. We use precisely the same construction as in the previous section dealing with the count of prime divisors. Thus, at this point, we know that an $n\in \Omega$ satisfies all of (0), (1), and (2) with positive probability. We now show that
\begin{enumerate}
\item[(3)] with probability $1+o(1)$, as $N\to\infty$, we have
\[ \varrho(n+h_i-1) - \omega(n+h_i-1) < (\log_2{N})(\log_4{N}) \] for all $1 \leq i \leq k$.
\end{enumerate}
Recall that for us, $\varrho(n)$ denotes the number of primes dividing $n$, counted with multiplicity. For any natural number $n$, we have \[ 2^{\omega(n)} \leq \tau(n) \leq 2^{\varrho(n)}. \] Using these bounds along with (0)--(3), it follows that $\tau(p-1)$ is strictly increasing on the set of primes between $n+h_1$ and $n+h_k$. The decreasing case can  be deduced by the now-familiar trick of swapping $\Hh$ for $-\Hh$.

Fix $1 \leq i \leq k$. Call an $n \in \Omega$ for which the inequality in (3) fails \emph{exceptional}. We will show that the number of exceptional $n \in \Omega$ is
\begin{equation}\label{eq:acceptableerror}\ll \frac{N}{W (\log{N})^A} \qquad\text{for each fixed $A$.} \end{equation}
From Lemma \ref{lem:csprep} and Cauchy--Schwarz, the probability of selecting such an $n$ is then
\[ \frac{1}{S_1}\sum_{\substack{n \in \Omega \\ n\text{ exceptional}}} w(n) \ll S_1^{-1} \left(\frac{N}{W (\log{N})^A}\right)^{1/2} \left(\frac{N}{W} (\log{R})^{19k}\right)^{1/2}, \]
which is $o(1)$ if we take $A > 19k$.

Write the prime factorization of $n+h_i-1$ in the form $\prod p^{e_p}$. If $n$ is exceptional, then
\begin{equation}\label{eq:ifexceptional0} \sum_{p\mid n+h_i-1} (e_p-1) \ge \log_2{N} \log_4{N}. \end{equation}
	
Let
\begin{equation}\label{eq:Bdef} B\coloneqq \lfloor\frac{1}{4} \log_2{N} \log_4{N}\rfloor. \end{equation}

We can suppose that
\begin{equation}\label{eq:ifexceptional1} e_2-1 < B.\end{equation} For if $e_2 \ge B+1$, then
\[ n \equiv \nu \pmod{W} \quad\text{and}\quad n \equiv 1-h_i \pmod{2^{B+1}}, \]
so that $n$ belongs to a fixed residue class modulo $[W,2^{B+1}]$. (Remember that $i$ is \emph{fixed} at this stage of the argument.) The number of such $n \in [N,2N)$ is
\[\le \frac{N}{[W,2^{B+1}]}+1 \leq \frac{N}{W \exp(\frac{1}{8} \log_2{N} \log_4{N})},\]
for large $N$. This fits within our desired bound \eqref{eq:acceptableerror} on the size of the exceptional set, since $\exp(\frac{1}{8} \log_2{N} \log_4{N})$ grows faster than $(\log{N})^A$ for any fixed $A$.

We can also suppose that
\begin{equation}\label{eq:ifexceptional2}\sum_{p \in \Pp_i} (e_p-1) < B.\end{equation} To see this, let $P \coloneqq \prod_{p \in \Pp_i}p$. By our choice of $\nu$, we have that $P \mid n+h_i-1$. Let \[ n' = (n+h_i-1)/P.\] If $\sum_{p \in \Pp_i} (e_p-1) \ge B$, then $r \coloneqq \prod_{p \in \Pp_i} p^{e_p-1}$ is a divisor of $n'$ with $\varrho(r) \ge B$. Replacing $r$ with one of its proper divisors if necessary, we obtain a divisor $r$ of $n'$ supported on the primes in $\Pp_i$ and having $\varrho(r)=B$. Let us count the possibilities for $n'$ given $r$. Since $n+h_i-1 \le 3N$, we have $n' \le 3N/P$. Since $n\equiv \nu \pmod{W}$, the integer $n'$ is uniquely determined modulo $W/P$. Since $r \mid n'$ and $\gcd(r,W/P)=1$, we see that $n'$ lies in a fixed residue class modulo $rW/P$. Hence, the number of possibilities for $n'$, given $r$, is at most $\frac{3N/P}{rW/P} + 1 = \frac{3N}{rW}+1$. Now sum over the possible values of $r'$. We find that the total number of $n'$ that can arise this way is
\begin{equation}\label{eq:nprimetotal} \ll \frac{N}{W} \sum_{\substack{p \mid r \Rightarrow p \in \Pp_i \\ \varrho(r) = B}} \frac{1}{r} + (\#\Pp_i)^{B}. \end{equation}
Recalling our choices of $\#\Pp_i$ and $B$,
\[ (\#\Pp_i)^{B}\le (k (\log_2 N) \log_3 N)^{\log_2{N} \log_4{N}} = N^{o(1)}, \]
while (crudely)
\[ \sum_{\substack{p \mid r \Rightarrow p \in \Pp_i \\ \varrho(r) = B}}\frac{1}{r}  \leq \left(\sum_{p \in \Pp_i}\frac{1}{p}\right)^{B} \leq \left(\frac{4(\log_2{N})^2}{\log{N}} \cdot (k \log_2{N} \log_3{N}) \right)^{B} \leq \exp(-(\log_2{N})^2) \]
for large $N$. Plugging these estimates back into \eqref{eq:nprimetotal}, we see that $n'$ is restricted to a set of size
\[ \ll \frac{N}{W \exp((\log_2{N})^2)} + N^{o(1)}, \]
which fits within the bound \eqref{eq:acceptableerror}. Since $n'$ determines $n$, we may indeed assume that $\sum_{p \in \Pp_i} (e_p-1) < B$.

Since the only primes dividing $W$ and $n+h_i-1$ are $2$ and the primes in $\Pp_i$, combining \eqref{eq:ifexceptional0}--\eqref{eq:ifexceptional2} yields
\[ \sum_{\substack{p\mid n+h_i-1 \\ p \nmid W}} (e_p-1) \geq \frac{1}{2}\log_2{N} \log_4{N} \ge 2B. \]
Define
\[ s \coloneqq\prod_{\substack{p \mid n+h_i-1 \\ e_p \geq 2 \\ p\nmid W}} p^{e_p}. \]
Clearly, $s$ is squarefull. Moreover,
\[ \varrho(s) = \sum_{\substack{p \mid n+h_i-1 \\ e_p \geq 2 \\ p\nmid W}} e_p \ge \sum_{\substack{p \mid n+h_i-1 \\ p\nmid W}} (e_p-1) \ge 2B, \] so that $s \ge 2^{2B}$. Since $s$ divides $n+h_i-1$ and $\gcd(s,W)=1$, we see that $n$ is determined modulo $Ws$. Hence, the number of possibilities for $n$, given $s$, is $O(\frac{N}{Ws} + 1)$. Summing over all squarefull $s \in (2^{2B}, 3N]$ gives an upper bound on the number of these $n$ that is
\[ \ll \frac{N}{W} \exp\big(-\frac{1}{8}\log_2{N} \log_4{N}\big) + N^{1/2}, \]
which fits comfortably within the bound \eqref{eq:acceptableerror}. This completes the proof.

\section{Sums of digits of primes}
We need yet another variant of the Maynard--Tao theorem, namely Theorem 1 of \cite{BFTB14}, due to Banks, Freiberg, and Turnage--Butterbaugh.

\begin{prop}\label{prop:maynardsoft} Let $K$ and $k$ be positive integers with $K \geq 2$ and $k \geq e^{8K+5}$. Let $\Hh = \{h_1, \dots, h_k\}$ be an admissible $k$-tuple, and let $A\geq 1$ be a fixed integer coprime to $h_1 \cdots h_k$. Then for some subset $\{b_1, \dots, b_K\} \subset \{h_1, \dots, h_k\}$, there are infinitely many $n \in \N$ such that $An + b_1, \dots, An + b_K$ are consecutive primes.
\end{prop}

It is possible to derive affirmative answers to Sierpi\'nski's questions directly from Proposition \ref{prop:maynardsoft}, making use of ad hoc elementary constructions. However, we feel it is more enlightening to appeal to a deep recent result of Drmota, Mauduit, and Rivat \cite[Theorem 1.1]{DMR09}.

\begin{prop}\label{prop:DMR} Fix $g\ge 2$. Let $\mu_g\coloneqq \frac{g-1}{2}$ and $\sigma_g^2 \coloneqq \frac{g^2-1}{12}$. For each integer $\ell \ge 0$ with $\gcd(\ell,g-1)=1$, we have
	\[ \#\{p \leq x: s_g(p)= \ell\} = \frac{g-1}{\phi(g-1)} \frac{\pi(x)}{\sqrt{2\pi \sigma_g^2\cdot \frac{\log{x}}{\log{g}}}} \left(e^{-\frac{(\ell - \mu_g \log{x}/\log{g})^2}{2\sigma_g^2 \log{x}/\log{g}}} + O_{\epsilon,g}((\log{x})^{-\frac{1}{2}+\epsilon})\right), \]
	where $\epsilon > 0$ is arbitrary but fixed.
\end{prop}

\begin{proof}[Proof that there are runs of arbitrary length on which $s_g(p)$ is constant] Let the length $K$ be given, and let $k$ be the smallest integer exceeding $e^{8K+5}$.  If $x$ is large and $\ell$ is chosen to be the nearest integer to $\mu_g\log{x}/\log{g}$ that is coprime to $g-1$, then Proposition \ref{prop:DMR} ensures that the count of $p \leq x$ with $s_g(p) = \ell$ is $\gg_{g} x/(\log{x})^{3/2}$. So for a sufficiently large  fixed choice of $x$, we will get at least $k$ such values of $p$, each of which satisfies $p > \max\{g, k\}$. Pick $k$ of these primes, say $p_1, p_2, \dots, p_k$. Since $p_1, \dots, p_k$ are distinct primes exceeding $k$, it is clear that $\{p_1, \dots, p_k\}$ is an admissible $k$-tuple. Now apply Proposition \ref{prop:maynardsoft} with $A=g^N$, where $N$ is chosen so that $g^N$ is larger than each of $p_1, \dots, p_k$. We obtain $K$  consecutive primes each of which has the same sum of digits.
\end{proof}

\begin{proof}[Proof that there are arbitrarily long increasing and decreasing runs] Let $K$ be given, and let $k$ be the smallest integer exceeding $e^{8K+5}$. 	
With $x$ large, we let $\ell_1 < \ell_2 < \dots < \ell_k$ be the smallest integers exceeding $\mu_g \log{x}/\log{g}$ that are coprime to $g-1$. Let $x_j \coloneqq 2^{j-1} x$ for $1 \leq j \leq k$. Then each $\ell_i = \mu_g\log{x}/\log{g} + O_{g,K}(1)$, and each $\log{x_j} = \log{x} + O_{K}(1)$. So from Proposition \ref{prop:DMR}, for each pair of $i$ and $j$, there is a prime $p \in [x_j, 2x_j)$ with $s(p) = \ell_j$. By choosing $x$ large enough, we guarantee that $p > \max\{g, k\}$.

We apply this to pick primes $p_1, \dots, p_k$ with each $p_i \in [x_i, 2x_i)$ and $s_g(p_i)=\ell_i$. Then $p_1 < p_2 < \dots < p_k$ and $s_g(p_1) < s_g(p_2) < \dots < s_g(p_k)$. Now apply Proposition \ref{prop:maynardsoft} with $\Hh=\{p_1, \dots, p_k\}$ and $A = g^N$, where $g^N$ is larger than $\max\{p_1, \dots, p_k\}$. This gives infinitely many runs of $K$ consecutive primes on which $s_g$ is increasing. To get decreasing runs, we proceed in the same way but choose $p_i \in [x_i, 2x_i)$ so that $s_g(p_i)=\ell_{k+1-i}$.
\end{proof}

\section*{Acknowledgements}
We thank Bill Banks and Tristan Freiberg for their assistance with the proof of Proposition \ref{prop:mm2}. We also thank the referee for feedback that led to improvements in the exposition. The first author is supported by NSF award DMS-1402268. The second author is supported by an AMS Simons Travel Grant.

\providecommand{\bysame}{\leavevmode\hbox to3em{\hrulefill}\thinspace}
\providecommand{\MR}{\relax\ifhmode\unskip\space\fi MR }
\providecommand{\MRhref}[2]{%
  \href{http://www.ams.org/mathscinet-getitem?mr=#1}{#2}
}
\providecommand{\href}[2]{#2}

\end{document}